\documentclass{amsart}
\usepackage{amssymb}
\usepackage[all]{xy}

\newcommand{\IR}{\mathbb R}
\newcommand{\IQ}{\mathbb Q}

\newcommand{\w}{\omega}
\newcommand{\U}{\mathcal U}
\newcommand{\F}{\mathcal F}
\newcommand{\V}{\mathcal V}
\newcommand{\C}{\mathcal C}
\newcommand{\A}{\mathcal A}
\newcommand{\M}{\mathcal M}
\newcommand{\K}{\mathcal K}
\newcommand{\add}{\mathrm{add}}
\newcommand{\supp}{\mathrm{supp}}
\newcommand{\pr}{\mathrm{pr}}
\newcommand{\Ra}{\Rightarrow}
\newcommand{\cbox}{\boxdot}
\newcommand{\cl}{\mathrm{cl}}
\newcommand{\St}{\mathcal{S}t}
\newcommand{\id}{\mathrm{id}}
\newcommand{\e}{\varepsilon}
\newcommand{\diam}{\mathrm{diam}}

\newtheorem{theorem}{Theorem}[section]
\newtheorem{corollary}[theorem]{Corollary}
\newtheorem{proposition}[theorem]{Proposition}
\newtheorem{problem}[theorem]{Problem}
\newtheorem{lemma}[theorem]{Lemma}
\newtheorem{example}[theorem]{Example}
\theoremstyle{definition}
\newtheorem{definition}[theorem]{Definition}
\newtheorem{remark}[theorem]{Remark}

\title[Maslyuchenko spaces]{Quasicontinuous and separately continuous functions\\ with values in Maslyuchenko spaces}
\author{Taras Banakh}
\address{Ivan Franko National University of Lviv (Ukraine) and Jan Kochanowski University in Kielce (Poland)}
\email{t.o.banakh@gmail.com}
\thanks{The author has been partially financed by NCN grant DEC-2012/07/D/ST1/02087.}
\dedicatory{Dedicated to the 65-th birthday of V.K.~Maslyuchenko}

\begin{document}

\begin{abstract} We generalize some classical results about quasicontinuous and separately continuous functions with values in metrizable spaces to functions with values in certain generalized metric spaces, called Maslyuchenko spaces. We establish stability properties of the classes of Maslyuchenko spaces and study the relation of these classes to known classes of generalized metric spaces (such as Piotrowski or Stegall spaces). One of our results says that for  any $\aleph_0$-space $Z$ and any separately continuous function $f:X\times Y\to Z$ defined on the product of a topological space $X$ and a second-countable space $Y$, the set $D(f)$ of discontinuity points of $f$ has meager projection on $X$.
\end{abstract}
\maketitle

\section{Introduction}

The problem of evaluation of the sets of discontinuity points of quasicontinuous and separately continuous functions is classical in Real Analysis and traces its history back to the famous dissertation of Baire \cite{Baire}.

Let us recall that a function $f:X\to Y$ between topological spaces is {\em quasicontinuous at a point} $x\in X$ if for any neighborhood $O_x\subset X$ of $x$ and any neighborhood $O_{f(x)}\subset Y$ of $f(x)$ there exists a non-empty open set $U\subset O_x$ such that $f(U)\subset O_{f(x)}$. A function $f:X\to Y$ is {\em quasicontinuous} if it is quasicontinuous at each point $x\in X$. Formally, quasicontinuous functions were introduced by Kempisty \cite{Kemp} but implicitly they appeared earlier in works of Baire and Volterra.

The following property of quasicontinuous functions is well-known (and can be easily derived from the definition), see \cite{Bled}, \cite{Lev63}, \cite{Neu}, \cite{Piot85}, \cite{Piot90}.

\begin{theorem}\label{t:K} For any quasicontinuous function $f:X\to Y$ from a topological space $X$ to a metrizable space $Y$ the set $D(f)$ of discontinuity points of $f$ is meager in $X$.
\end{theorem}

We recall that a subset $M$ of a topological space $X$ is {\em meager} in $X$ if $M$ can be written as the countable union of nowhere dense subsets of $X$.

In fact, the metrizability of the space $Y$ in Theorem~\ref{t:K} can be weakened to the strict fragmentability of $Y$. We recall that a topological space $Y$ is {\em fragmented} by a metric $d$ if for every $\e>0$, each non-empty subspace $A\subset X$ contains a non-empty relatively open subset $U\subset A$ of $d$-diameter $\diam(U)<\e$. If the metric $d$ generates a topology at least as strong as the original topology of $X$, then we shall say that $X$ is {\em strictly fragmented} by the metric $d$. A topological space is called ({\em strictly}) {\em fragmentable} if it is (strictly) fragmented by some metric. It is clear that each metrizable space is strictly fragmentable and each strictly fragmentable space is fragmentable. By \cite{Rib}, a compact Hausdorff space is strictly fragmentable if and only if it is fragmentable. In \cite{KKM} this equivalence was generalized to so-called game determined spaces. The following  generalization of Theorem~\ref{t:K} was obtained in \cite[Theorem 2.7]{GKMS}, \cite[Lemma 4.1]{MG}, \cite[Theorem 5.1]{KM3}, \cite[Theorem 1]{KKM}.

\begin{theorem}[Giles-Kenderov-Kortezov-Moors-Sciffer]\label{t:sF-P} For any quasicontinuous function $f:X\to Y$ from a topological space $X$ to a strictly fragmentable space $Y$ the set $D(f)$ of discontinuity points of $f$ is meager in $X$.
\end{theorem}

In \cite{Piot98} Piotrowski suggested to study spaces $Y$ for which every quasicontinuous function $f:X\to Y$ defined on a non-empty Baire space $X$ has a continuity point. Topological spaces $Y$ possessing this property are called {\em Piotrowski} spaces. Properties of Piotrowski spaces are discussed in the survey paper  \cite{BM}. By \cite{BM}, a regular topological space $X$ is Piotrowski if and only if for any quasicontinuous function $f:Z\to X$ the set $D(f)$ is discontinuity points of $f$ is meager in $Z$.

Concerning discontinuity points of separately continuous functions we have the following classical result due to Calbrix and Troallic \cite{CT}.

\begin{theorem}[Calbrix-Troallic]\label{t:CT} Let $X$ be a topological space, $Y$ be a second countable space and $Z$ be a metrizable space. For any separately continuous function $f:X\times Y\to Z$ the set $D(f)$ of discontinuity points of $f$ has meager projection on $X$.
\end{theorem}

In \cite{M98} this theorem was generalized to $KC$-functions.
Let $X,Y,Z$ be topological spaces. A function $f:X\times Y\to Z$ is called a {\em $KC$-function} if
\begin{itemize}
\item for every $y\in Y$ the function $f_y:X\to Z$, $f_y:x\mapsto f(x,y)$, is quasicontinuous, and
\item for every $x\in X$ the function $f^x:Y\to Z$, $f^x:y\mapsto f(x,y)$, is continuous.
\end{itemize}

Generalizing the Calbrix-Troallic Theorem, Maslyuchenko \cite{M98} obtained the following result.

\begin{theorem}[Maslyuchenko]\label{t:KC} Let $X$ be a topological space, $Y$ be a second countable space and $Z$ be a metrizable space. For any $KC$-function $f:X\times Y\to Z$ the set $D(f)$ of discontinuity points has meager projection on $X$.
\end{theorem}

In \cite{MN00} Theorem~\ref{t:KC} was generalized to $K_hC$-function and in \cite{BT} it was further generalized to functions $f:X\times Y\to Z$ which are lower quasicontinuous with respect to the first variable and continuous with respect to the second variable. In this paper such functions will be called $QC$-functions.

More precisely, a function $f:X\times Y\to Z$ is called a {\em $QC$-function} if
\begin{itemize}
\item $f$ is {\em lower quasicontinuous with respect to the first variable} in the sense that for every point $(x,y)\in X\times Y$ and open sets $O_x\subset X$, $O_y\subset Y$, $O_{f(x,y)}\subset Z$  with $(x,y)\in O_x\times O_y$ and $f(x,y)\in O_{f(x,y)}$ there exists a non-empty open set $W\subset O_x$ such that $f(\{w\}\times O_y)\cap O_{f(x,y)}\ne\emptyset$ for every $w\in W$, and
\item $f$ is {\em continuous with respect to the second variable}, which means that for every $x\in X$ the function $f^x:Y\to Z$, $f^x:y\mapsto f(x,y)$, is continuous.
\end{itemize}

The following generalization of Theorem~\ref{t:KC} was proved by Bouziad and Troallic \cite{BT}. It was further generalized by Maslyuchenko and Nesterenko \cite{MN14}.

\begin{theorem}[Bouziad-Troallic]\label{t:QC} Let $X$ be a topological space, $Y$ be  a second  countable space and $Z$ be a metrizable space. For every $QC$-function $f:X\times Y\to Z$ the set $D(f)$ of discontinuity points of $X$ has meager projection on $X$.
\end{theorem}

The Calbrix-Troallic Theorem implies the following known description of the set $D(f)$ of discontinuity points of  a separately continuous function $f$ defined on the product of second-countable spaces.

\begin{theorem}\label{t:CC} Let $X,Y$ be second-countable spaces and $Z$ be a metrizable space. For any separately continuous function $f:X\times Y\to Z$ the set $D(f)$ of discontinuity points of $f$ has meager projections on the spaces $X$, $Y$, and hence is contained in the product $M_X\times M_Y$ of two meager sets $M_X\subset X$ and $M_Y\subset Y$.
\end{theorem}

In \cite{M92} V.K.Maslyuchenko initiated a program of research with aim to extend known results about quasicontinuous or separately continuous functions with values in metrizable spaces to functions taking values in more general spaces. Later this program was realized in the papers \cite{M99, MMF06, MMF08,MMS} of Maslyuchenko and coauthors.

%A similar problem was posed by Piotrowski \cite{Piot98} who suggested to study spaces $Y$ for which every quasicontinuous function $f:X\to Y$ defined on a Baire space $X$ has a continuity point. This problem of Piotrowski was discussed by Kenderov, Kortezov and Moors  in \cite{KKM}.

Motivated by the Maslyuchenko program, in this paper we suggest a global approach to generalizing Theorems~\ref{t:CT}--\ref{t:CC} and introduce the following three classes of topological spaces.

\begin{definition} A topological space $Z$ is called
\begin{itemize}
\item {\em $KC$-Maslyuchenko}  if for each $KC$-function $f:X\times Y\to Z$ defined on the product of a topological space $X$ and a second-countable space $Y$ the set $D(f)\subset X\times Y$ has meager projection on $X$;
\item {\em $QC$-Maslyuchenko}  if for each $QC$-function  $f:X\times Y\to Z$ defined on the product of a topological space $X$ and a second-countable space $Y$ the set $D(f)\subset X\times Y$ has meager projection on $X$;
\item {\em $CC$-Maslyuchenko} if for any separately continuous function $f:X\times Y\to Z$ defined on the product of separable metrizable spaces $X,Y$ the set $D(f)\subset X\times Y$ has meager projections on $X$ and $Y$.
%\item {\em Piotrowski} if for each quasicontinuous function $f:X\to Z$ defined on a topological space $X$ the set $D(f)$ of discontinuity points of $f$ is meager in $X$.
\end{itemize}
\end{definition}

For any topological spaces these properties relate as follows:
$$
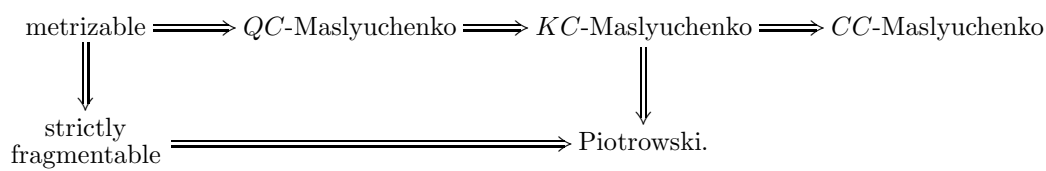
\xymatrix{
\mbox{metrizable}\ar@{=>}[r]\ar@{=>}[d]&\mbox{$QC$-Maslyuchenko}\ar@{=>}[r]&\mbox{$KC$-Maslyuchenko}\ar@{=>}[r]\ar@{=>}[d]&\mbox{$CC$-Maslyuchenko}\\
\mbox{strictly}\atop\mbox{fragmentable}\ar@{=>}[rr]&&\mbox{Piotrowski}.
}
$$

The classes of $QC$-Maslyuchenko, $KC$-Maslyuchenko, $CC$-Maslyuchenko spaces will be studied in Sections~\ref{s:QC}, \ref{s:KC}, \ref{s:CC}, respectively.   In Section~\ref{s:stable} we introduce $M$-stable classes of topological spaces and prove that each $M$-stable class of topological spaces contains all $\aleph_0$-spaces. In the next sections we prove that the classes of $QC$-Maslyuchenko, $KC$-Maslyuchenko, and $CC$-Maslyuchenko spaces are $M$-stable. Using the $M$-stability of the classes of Maslyuchenko spaces,  we prove that the class of $QC$-Maslyuchenko spaces contains all $\aleph_0$-spaces and all strongly $\sigma$-metrizable spaces, and the class of $CC$-Maslyuchenko spaces contains  and all $\aleph_0$-monolithic spaces (in particular, all $k^*$-metrizable spaces).

Now we recall the definitions of generalized metric spaces (mentioned above).

Let $X$ be a topological space. A family $\mathcal N$ of subsets of $X$ is called
\begin{itemize}
\item a {\em network} if for any point $x\in X$ and neighborhood $O_x\subset X$ of $x$ there is a set $N\in\mathcal N$ such that $x\in N\subset O_x$;
\item a {\em $k$-network} if for any compact subset $K\subset X$ and neighborhood $O_K\subset X$ of $K$ there exists a finite subfamily $\F\subset\mathcal N$ such that $K\subset \bigcup\F\subset O_x$.
\end{itemize}
It is clear that each base of the topology is a $k$-network and each $k$-network is a network.

A regular topological space $X$ is called
\begin{itemize}
\item {\em cosmic} if $X$ has a countable network;
\item an {\em $\aleph_0$-space} if $X$ has a countable $k$-network;
\item a {\em $\sigma$-space} if $X$ has a $\sigma$-locally finite network;
\item an {\em $\aleph$-space} if $X$ has a $\sigma$-locally finite $k$-network;
\item {\em $k^*$-metrizable} if $X$ has $\sigma$-compact-finite $k$-network;
\item {\em $\sigma$-metrizable} if $X$ has a countable cover by closed metrizable subspaces;
\item {\em strongly} {\em $\sigma$-metrizable} if $X$ can be written as the union $X=\bigcup_{n\in\w}X_n$ of an (increasing) sequence of closed metrizable subspaces such that each convergent sequence in $X$ is contained in some set $X_n$, $n\in\w$;
\item {\em $\aleph_0$-monolithic} if each countable subspace of $X$ is an $\aleph_0$-space and  the sequential closure of any $\aleph_0$-subspace in $X$ is an $\aleph_0$-space.
\end{itemize}
By the {\em sequential closure} of a subset $A$ of a topological space $X$ we understand the set $\cl_1(A)$ of limit points of sequences $\{a_n\}_{n\in\w}\subset A$ that converge in $X$.

More information on generalized metric spaces can be found in the surveys of Gruenhage \cite{Grue}, \cite{Grue2}, \cite{Grue3} (for  $k^*$-metrizable spaces, see \cite{BBK}). By Theorem 6.4 in \cite{BBK}, a regular space $Y$ is $k^*$-metrizable if and only if $Y$ is the image of a metrizable space $M$ under a continuous map $f:M\to X$ that admits a section $s:X\to M$ which preserves precompact sets in the sense that for every compact subset $K\subset X$ the set $s(K)$ has compact closure in $M$. This condition implies that the map $f:M\to X$ is {\em compact-covering} in the sense that every compact set $K\subset X$ coincides with the image $f(C)$ of some compact set $C\subset M$.
By \cite{Mich}, a regular space $X$ is cosmic (resp. an $\aleph_0$-space) if and only if $X$ is the image of a separable metrizable space under a continuous (and compact-covering) map. By \cite{BM}, each $\sigma$-space is strictly fragmentable.

The interplay between the generalized metric spaces discussed above are described in the following diagram.
$$
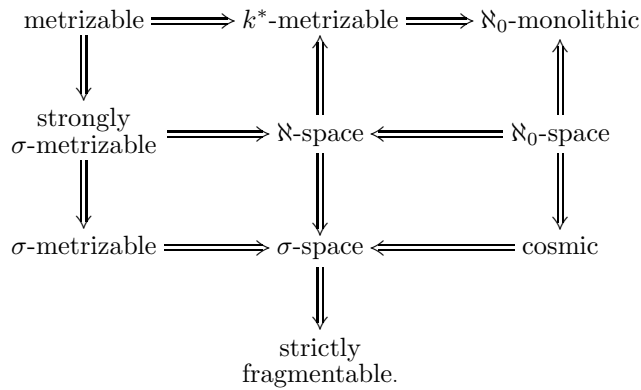
\xymatrix{
\mbox{metrizable}\ar@{=>}[r]\ar@{=>}[d]
&\mbox{$k^*$-metrizable}\ar@{=>}[r]&\mbox{$\aleph_0$-monolithic}\\
\mbox{strongly}\atop\mbox{$\sigma$-metrizable}\ar@{=>}[r]\ar@{=>}[d]
&\mbox{$\aleph$-space}\ar@{=>}[d]\ar@{=>}[u]&\mbox{$\aleph_0$-space}\ar@{=>}[l]\ar@{=>}[d]\ar@{=>}[u]\\
\mbox{$\sigma$-metrizable}\ar@{=>}[r]&\mbox{$\sigma$-space}\ar@{=>}[d]&\mbox{cosmic}\ar@{=>}[l]\\
&\mbox{strictly}\atop\mbox{fragmentable}.
}
$$
\smallskip

Adding to this diagram the classes of Piotrowski and Maslyuchenko spaces, we get the following diagram (describing the main results proved in this paper). In this diagram by a {\em separably $CC$-Maslyuchenko space} we understand a topological space whose every separable subspace is $CC$-Maslyuchenko.

$$
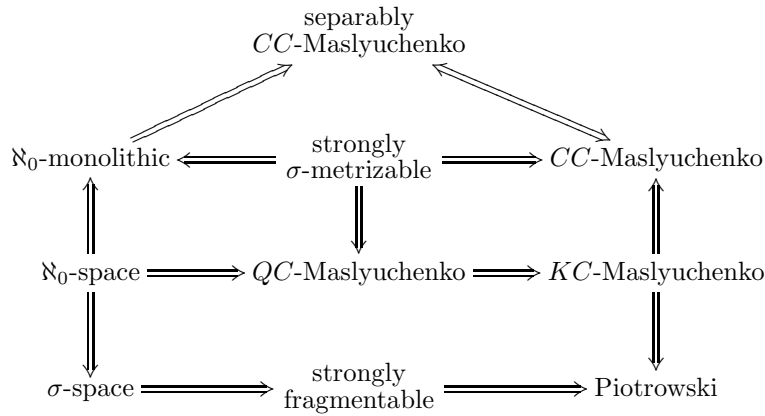
\xymatrix{
&\mbox{separably}\atop\mbox{$CC$-Maslyuchenko}\ar@{<=>}[rd]\\
\mbox{$\aleph_0$-monolithic}\ar@{=>}[ru]
&\mbox{strongly}\atop\mbox{$\sigma$-metrizable}\ar@{=>}[d]\ar@{=>}[r]\ar@{=>}[l]&
\mbox{$CC$-Maslyuchenko}\\
\mbox{$\aleph_0$-space}\ar@{=>}[r]\ar@{=>}[u]\ar@{=>}[d]&\mbox{$QC$-Maslyuchenko}\ar@{=>}[r]&\mbox{$KC$-Maslyuchenko}\ar@{=>}[d]\ar@{=>}[u]\\
%&\mbox{Eberlein compact}\ar@{=>}[u]\ar@{=>}[d]\\%\ar[l]|-{|}\\
\mbox{$\sigma$-space}\ar@{=>}[r]&\mbox{strongly}\atop\mbox{fragmentable}\ar@{=>}[r]&\mbox{Piotrowski}
%&\mbox{compact}\atop\mbox{fragmentable}\ar@{=>}[l]
%\\&\mbox{scatteredly}\atop\mbox{Piotrowski}\ar@{<=>}[ru]\\
}
$$
\smallskip

\smallskip

{\bf Acknowledgement.} The author would like to thank the anonymous referee of an initial version of this paper for valuable remarks, suggestions, and references, which resulted in splitting the initial paper into two parts. This paper is the second part of the splitted paper; the first part \cite{BM} is devoted to Piotrowski spaces and its results are used in this paper.

\section{$M$-stable classes of topological spaces}\label{s:stable}

In this section we discuss some stability properties of classes of topological spaces and introduce $M$-stable classes of topological spaces. In the next sections we shall prove that the  classes of Maslyuchenko spaces are $M$-stable.

One of operations preserving the classes of Maslyuchenko spaces is the operation of a Michael modification. A typical example of a Michael modification is the Michael line $\IR_{\IQ}$. This is the real line $\IR$ endowed with the topology $\tau=\{U\subset\IR:U\cap\IQ$ is open in $\IQ\}$. The Michael line is known in General Topology as an example of a hereditarily paracompact space with non-normal product $\IR_{\IQ}\times(\IR\setminus \IQ)$ with irrationals (see \cite[5.1.32]{En}). The Michael line is submetrizable but fails to be a $\sigma$-space. The construction of the Michael line can be generalized as follows.

Let $g:L\to X$ be a continuous map from a locally compact paracompact space $(L,\tau_L)$ into a topological space $(X,\tau_X)$ such that $L\cap (X\setminus g(L))=\emptyset$. The set $X_g=L\cup (X\setminus g(L))$ endowed with the topology $\tau$ generated by the base $\tau_L\cup\{g^{-1}(U)\cup (U\setminus g(L)):U\in\tau_X\}$ is called {\em the Michael modification} of the space $X$ by the map $g$.

A closed subset $A$ of a topological space $X$ is called
\begin{itemize}
\item a {\em $\bar G_\delta$-set} in $X$ if $A=\bigcap_{n\in\w}\bar U_n$ for some sequence $(U_n)_{n\in\w}$ of open sets in $X$;
\item a {\em retract} of $X$ if there exists a continuous map $r:X\to A$ such that $r(a)=a$ for all $a\in A$;
\item a {\em neighborhood retract} in $X$ is $A$ is a retract of some closed neighborhood $\bar U$ of $A$ in $X$.
\end{itemize}

We shall say that a class $\C$ of topological spaces is closed under taking
\begin{itemize}
\item {\em  $\bar\sigma$-sums} if a topological space $X$ belongs to the class $\C$ provided $X=\bigcup_{n\in\w}X_n$ for some sequence $(X_n)_{n\in \w}$ of closed subsets $X_n\subset X$ that belong to the class $\C$;
\item {\em  strong $\bar\sigma$-sums} if a topological space $X$ belongs to the class $\C$ provided $X=\bigcup_{n\in\w}X_n$ for some increasing sequence $(X_n)_{n\in \w}$ of closed subsets $X_n\in\C$ of $X$ such that each convergent sequence in $X$ is contained in some set $X_n$;
\item {\em $\bar G_\delta$-retral unions} if a topological space $X$ belongs to the class $\C$ provided $X=X_1\cup\dots\cup X_n$ for some $\bar G_\delta$-sets $X_1,\dots,X_n\subset X$ that belong to the class $\C$ and are neighborhood retracts in $X$;
\item {\em fibered preimages} if a topological space $X$ belongs to the class $\C$ whenever $X$ admits a continuous map $f:X\to Y$ into a paracompact space $Y\in\C$ such that each point $y\in Y$ has a neighborhood $O_y\subset Y$ such that $f^{-1}(O_y)$ belongs to the class $\C$;
%\item {\em $\C$-scattered spaces} if each $\C$-scattered space belongs to the class $\C$;
\item {\em Michael modifications} if $\C$ contains the Michael modification $X_g$ of any Hausdorff space $X\in\C$ by a map $g:L\to X$ defined on a locally compact paracompact space $L\in\C$.
\end{itemize}

%Now we can define $P$-stable classes of topological spaces.

%\begin{definition} A class $\C$ to topological spaces is called {\em $P$-stable} if it contains all metrizable spaces and is closed under taking subspaces, homeomorphic images, countable Tychonoff products, fibered preimages, $\bar\sigma$-sums, Michael modifications, and $\C$-scattered spaces.
%\end{definition}

%In Section~\ref{s:K} we shall prove that the class of Piotrowski spaces is $P$-stable.

%Observe that $P$-stable classes are local in the following sense.

%\begin{proposition}\label{p:local} Assume that a class $\C$ of topological spaces is closed under taking subspaces, topological sums and $\C$-scattered spaces. A topological space $X$ belongs to the class $\C$ if and only if $X$ admits a cover by open subspaces that belong to the class $\C$.
%\end{proposition}

%\begin{proof} The ``only if'' part is trivial. To prove the ``if'' part, assume that $X$ has a cover $\U$ by open subspaces that belong to the class $\C$. We claim that the space $X$ is $\C$-scattered. Given any non-empty set $A\subset X$ we need to find a non-empty relatively open subset $V\subset A$ that belongs to the class $\C$. Take any point $a\in A$ and find a set $U\in\U$ that contains $a$. Since $U\in\C$ and the class $\C$ is closed under taking subspaces, the set $U\cap A$ is a required non-empty open set in $A$ that belongs to the class $\C$. The space $X$, being $\C$-scattered, belongs to the class $\C$ (which is closed under taking $\C$-scattered subspces).
%\end{proof}

%\smallskip

The definition of an $M$-stable class includes also the stability under taking certain function spaces. For two topological spaces $X,Y$ by $C(X,Y)$ we denote the set of continuous functions from $X$ to $Y$.
Any family $\K$ of compact subsets of $X$ induces a topology $\tau_\K$ (called {\em the $\K$-open topology})  on $C(X,Y)$, which is generated by the subbase consisting of the sets $$[K,U]=\{f\in C(X,Y):f(K)\subset U\}$$where $K\in\K$ and $U$ is an open set in $Y$. The function space $C(X,Y)$ endowed with the $\K$-open topology will be denoted by $C_\K(X,Y)$. If the family $\K$ coincides with the family of all compact (finite) subsets of $X$, then the topology $\tau_\K$ will be denoted by $\tau_k$ (resp. $\tau_p$) and the function space $C_\K(X,Y)$ will be denoted by $C_k(X,Y)$ (resp. $C_p(X,Y)$).  For a subspace $Z\subset C(X,Y)$ by $\tau_p|Z$ and $\tau_k|Z$ we denote the topologies on $Z$ induced by the topologies $\tau_p$ and $\tau_k$, respectively.

A topology $\tau$ on a subset $Z\subset C(X,Y)$ is called
\begin{itemize}
\item {\em admissible} if $\tau_k|Z\subset \tau\subset\tau_p|Z$;
\item {\em sequentially admissible} if it is admissible and the evaluation function $e:Z\times X\to Y$, $e:(f,x)\mapsto f(x)$, is sequentially continuous.
\end{itemize}

We recall that a function $f:X\to Y$ between topological spaces is called {\em sequentially continuous} if for any sequence $(x_n)_{n\in\w}$ convergent to a point $x_\infty$ in $X$ the sequence $(f(x_n))_{n\in\w}$ converges to the point $f(x_\infty)$ in $Y$. It can be shown that a function $f:X\to Y$ is sequentially continuous if and only if for any compact metrizable subset $K\subset X$ the restriction $f|K$ is continuous.

The following lemma yields many examples of sequentially admissible topologies on function spaces.

\begin{lemma}\label{l:C_K} Let $X$ be a topological space and $\K$ be a family of compact subsets of $X$, which includes all compact sets of the form $\{\lim x_n\}\cup\{x_n\}_{n\in\w}$ where $(x_n)_{n\in\w}$ is a convergent sequence in $X$. For any topological space $Y$ the $\K$-open topology $\tau_\K$ on the function space $C(X,Y)$ is sequentially admissible.
\end{lemma}

\begin{proof} It is clear that $\tau_k\subset\tau_\K\subset\tau_p$. It remains to check that for any  function sequence $(f_n)_{n\in\w}$ convergent to a function $f$ in the function space $C_\K(X,Y)$ and any sequence $(x_n)_{n\in\w}$ convergent to a point $x$ in the space $X$, the sequence $\big(f_n(x_n)\big)_{n\in\w}$ converges to $f(x)$ in $Y$. Take any open neighborhood $O_{f(x)}\subset Y$ of $f(x)$. By the continuity of $f$, there is a number $n_0\in\w$ such that $\{f(x_n)\}_{n\ge n_0}\subset O_{f(x)}$. Since the compact set $K=\{x\}\cup\{x_n\}_{n\ge n_0}$ belongs to the family $\K$, the set $[K;O_{f(x)}]=\{g\in C(X,Y):g(K)\subset O_{f(x)}\}$ is an open neighborhood of $f$ in the function space $C_\K(X,Y)$. The convergence of the sequence $(f_n)_{n\in\w}$ to $f$ yields a number $n_1\ge n_0$ such that $\{f_n\}_{n\ge n_1}\subset [K,O_{f(x)}]$. Then for every $n\ge n_1$ we get $f_n(x_n)\in f_n(K)\subset O_{f(x)}$, which means that the sequence $(f_n(x_n))_{n\in\w}$ converges to $f(x)$.
\end{proof}

We shall say that a class $\C$ of topological spaces is {\em closed under taking sequentially admissible function spaces} if for any $\aleph_0$-space $X$ and space $Y\in\C$, every subspace $Z\subset C(X,Y)$ endowed with an sequentially admissible topology belongs to the class $\C$.

Now we are able to define $M$-stable classes of topological spaces.

\begin{definition} A class $\C$ of topological spaces is defined to be {\em $M$-stable} if
$\C$ contains all metrizable spaces and is closed under taking subspaces, homeomorphic images, countable Tychonoff products, fibered preimages, strong $\bar\sigma$-sums, $\bar G_\delta$-retral unions, Michael modifications, and sequentially admissible function spaces.
\end{definition}

Lemma~\ref{l:C_K} implies the following proposition.

\begin{proposition} If a class $\C$ is $M$-stable, then for each $\aleph_0$-space $X$, each family $\K$ of compact subsets containing  all non-empty compact subsets with at most one non-isolated point in $X$, and each space $Y\in\C$, the function space $C_\K(X,Y)$ belongs to the class $\C$.
\end{proposition}

The following important theorem yields a ``lower bound'' on $M$-stable classes.

\begin{theorem}\label{t:aleph0} Each $M$-stable class $\C$ of topological spaces contains all $\aleph_0$-spaces.
\end{theorem}

\begin{proof} We should prove that any $\aleph_0$-space $Z$ belongs to the class $\C$. By Michael's Theorem \cite{Mich} (see also \cite[11.5]{Grue}) the function space $C_k(Z):=C_k(Z,\IR)$ is an $\aleph_0$-space. Consider the canonical map $\delta_Z:Z\to C_k(C_k(Z))$ assigning to each point $z\in Z$ the Dirac measure concentrated at $z$. Taking into account that the space $Z$ is Tychonoff, we can prove that the map $\delta_Z$ is injective and the inverse map $\delta_Z^{-1}:\delta_Z(Z)\to Z$ is continuous. Consider the space $Z'=\delta_Z(Z)\subset C(C_k(Z))$ endowed with the topology turning the map  $\delta_Z:Z\to Z'$ into a homeomorphism. The continuity of the map $\delta_Z:Z\to C_p(C_k(X))$ implies that $\tau\subset\tau_p|Z'$ and the continuity of the map $\delta_Z^{-1}:(Z',\tau_k|Z')\to Z$ implies that $\tau_k|Z'\subset\tau$. Therefore, the topology $\tau$ on $Z'$ is admissible. Lemma~\ref{l:C_K} implies that the map $e:Z\times C_k(Z)\to\IR$, $e:(z,f)\mapsto f(z)$, is sequentially continuous, which means that the topology $\tau$ on $Z'$ is sequentially admissible.

 Taking into account that the class $\C$ contains the real line and $\C$ is closed under taking sequentially admissible function spaces, we conclude that the space $Z'\subset C(C_k(Z),\IR)$ endowed with the sequentially admissible topology $\tau$ belongs to the class $\C$. Since $\C$ is closed under homeomorphic images, the space $Z$ belongs to the class $\C$, too.
 \end{proof}

Now we establish some properties of classes that are closed under taking fibered preimages.

\begin{proposition}\label{p:topsum} If a class $\C$ of topological spaces contains all discrete spaces and is closed under taking fibered preimages, then $\C$ is closed under topological sums.
\end{proposition}

\begin{proof} Assume that a topological space $X$ is a topological sum $X=\bigoplus_{\alpha\in A}X_\alpha$ of clopen subspaces $X_\alpha\in\C$ of $X$. Endow the index set $A$ with the discrete topology and consider the continuous map $f:X\to A$ assigning to each point $x\in X$ the unique index $\alpha\in A$ such that $x\in X_\alpha$. By our assumption, the discrete space $A$ belongs to the class $\C$. Also for every $\alpha\in A$ the preimage $f^{-1}(\{\alpha\})=X_\alpha$ belongs to the class $\C$. Taking into account that the class $\C$ is closed under fibered preimages, we conclude that $X\in\C$.
\end{proof}

A topological space $X$ is called a {\em $k_\w$-space} if $X$ has a countable cover $\K$ by compact subspaces such that a subset $F\subset X$ is closed in $X$ if and only if for every $K\in\K$ the intersection $F\cap K$ is compact.

\begin{proposition}\label{p:komega} Assume that a class $\C$ of topological spaces contains all compact metrizable spaces and is closed under taking strong $\bar\sigma$-sums. Then $\C$ contains any cosmic $k_\w$-space.
\end{proposition}

\begin{proof}  Given a cosmic $k_\w$-space $X$ find a sequence $(K_n)_{n\in\w}$ of compact sets in $X$ generating the topology of $X$ in the sense that a subset $F\subset X$ is closed if and only if for every $n\in\w$ the intersection $F\cap K_n$ is compact. Replacing each set $K_n$ by the union $\bigcup_{i\le n}K_i$, we can assume that $K_n\subset K_{n+1}$ for all $n\in\w$. Each compact space $K_n$, being cosmic, is metrizable and hence belongs to the class $\C$. It can be shown that each convergent sequence in $X$ is contained in some compact set $K_n$.
Then the class $\C$, being closed under strong $\bar\sigma$-sums, contains the space $X$.
\end{proof}

A cover $\U$ of a set $X$ is called {\em star-countable} if for each set $U\in\U$ the family $\{V\in\U:V\cap U\ne\emptyset\}$ is at most countable.

\begin{proposition}\label{p:star} Assume that a class $\C$ of topological spaces contains all locally compact metrizable spaces and is closed under taking subspaces, strong $\bar\sigma$-sums, and fibered preimages. A paracompact space $X$ belongs to the class $\C$ if and only if $X$ has a star-countable cover by open subspaces that belong to the class $\C$.
\end{proposition}

\begin{proof} The ``only if'' part is trivial. To prove the ``if'' part, assume that $\U$ is a star-countable cover of $X$ by open subspaces that belong to the class $\C$. Let $N(\U)$ be the nerve of the cover $\U$. This is a simplicial complex whose vertices are elements of the cover $\U$ and a finite subset $\F\subset\U$ is a simplex of $N(\U)$ iff $\bigcap\F\ne\emptyset$.
By the paracompactness of $X$, there is a partition of the unit $(\lambda_U:X\to[0,1])_{U\in\U}$ subordinated to the cover $\U$. This means that $\lambda_{U}^{-1}\big((0,1]\big)\subset U$ for all $U\in\mathcal U$, $\big(\lambda_U^{-1}((0,1])\big)_{U\in\U}$ is a locally finite cover of $X$, and $\sum_{U\in\U}\lambda_U\equiv 1$. The family $(\lambda_U)_{U\in\U}$ determines a continuous map $\lambda:X\to Y$ to the geometric realization $Y$ of the nerve $N(\U)$. Since the cover $\U$ is star-countable, the space $Y$ is a topological sum of the geometric realizations of countable simplicial complexes. Since the geometric realization of any countable simplicial complex is a cosmic $k_\w$-space, we can apply Propositions~\ref{p:topsum} and \ref{p:komega} to conclude that the class $\C$ contains the space $Y$. Known properties of the canonical map $\lambda:X\to Y$ (see \cite[\S3.9]{Sak}) ensure that the space $Y$ has an open cover $\V$ such that for every $V\in\V$ the preimage $\lambda^{-1}(V)$ is contained in some set $U\in\U$ and hence belongs to the class $\C$. Since the class $\C$ is closed under fibered preimages, the space $X$ belongs to the class $\C$.
\end{proof}

We recall that a topological space $X$ is {\em strongly paracompact} if each open cover of $X$ has a locally finite star-countable refinement.

\begin{corollary}\label{c:cover} Assume that a class $\C$ of topological spaces contains all locally compact metrizable spaces and is closed under taking subspaces, strong $\bar\sigma$-sums, and fibered preimages. A strongly paracompact space $X$ belongs to the class $\C$ if and only if $X$ has a cover by open subspaces that belong to the class $\C$.
\end{corollary}

We finish this section by proving that $M$-stable classes of topological spaces are closed under taking countable small-box products of $T_1$-pointed topological spaces. By a {\em pointed space} we understand a topological space $X$ with a distinguished point $*\in X$. A pointed space $(X,*)$ is called {$T_1$-pointed} if the singleton $\{*\}$ is closed in $X$. The {\em box-product} $\square_{\alpha\in A}X_\alpha$ of a family $(X_\alpha)_{\alpha\in A}$ of topological spaces is the Cartesian product $\prod_{\alpha\in A}X_\alpha$ endowed with the topology generated by the base consisting of the products $\prod_{\alpha\in A}U_\alpha$ of open sets $U_\alpha\subset X_\alpha$, $\alpha\in A$.
 For a family $(X_\alpha,*_\alpha)$, $\alpha\in A$, of pointed spaces the subspace
$$\cbox_{\alpha\in A}(X_\alpha,*_\alpha)=\{(x_\alpha)_{\alpha\in A}\in\square_{\alpha\in A}X_\alpha:\{\alpha\in A:x_\alpha\ne *_\alpha\}\mbox{ is finite}\}$$of $\square_{\alpha\in A}X_\alpha$ is called the {\em small box-product} of pointed spaces $(X_\alpha,*_\alpha)$, $\alpha\in A$.

We shall say that a topological space $X$ is {\em closed under taking small box-products} if for any sequence $(X_n,*_n)_{n\in\w}$ of $T_1$-pointed spaces that belong to the class $\C$ the small box-product $\cbox_{n\in\w}(X_n,*_n)$ belongs to the class $\C$.

\begin{proposition}\label{p:box} If a class $\C$ of topological spaces is closed under finite Tychonoff products and strong $\bar\sigma$-sums, then it is closed under taking small box-products.
\end{proposition}

\begin{proof} Take any sequence of $T_1$-pointed spaces $(X_n,*_b)$, $n\in\w$, that belong to the class $\C$. For every $m\in\w$ identify the finite product $\prod_{n<m}X_n$ with the closed subset
$$\{(x_n)\in\square_{n\in\w}X_n:\forall n\ge m\;\;x_n\ne *_n\}$$ of the small box-product $\cbox_{n\in\w}(X_n,*_n)$. It is standard to show that each convergent sequence in $\cbox_{n\in\w}X_n$ is contained in some finite product $\prod_{n<m}X_n$. Since the class $\C$ is closed under finite products, each space $\prod_{n<m}X_n$, $m\in\w$, belongs to the class $\C$. Since $\C$ is closed under taking strong $\bar\sigma$-sums, the small box-product $\cbox_{n\in\w}(X_n,*_n)=\bigcup_{m\in\w}\prod_{n<m}X_n$ belongs to the class $\C$.
\end{proof}

Propositions~\ref{p:topsum}, \ref{p:komega}, \ref{p:star}, \ref{p:box}, and Corollary \ref{c:cover} imply:

\begin{corollary} Assume that a class $\C$ of topological spaces is $M$-stable. Then
\begin{enumerate}
\item $\C$ contains all cosmic $k_\w$-spaces and is closed under taking topological sums and small box-products;
\item $\C$ contains a paracompact space $X$ if and only if $X$ admits a star-countable cover by open subspaces that belong to the class $\C$;
\item  $\C$ contains a strongly paracompact space $X$ if and only if $X$ admits a cover by open subspaces that belong to the class $\C$.
\end{enumerate}
\end{corollary}

\section{$QC$-Maslyuchenko spaces}\label{s:QC}

In this section we study $QC$-Maslyuchenko spaces and show that the class of such spaces is $M$-stable. We remind that a topological space $Z$ is called {\em $QC$-Maslyuchenko} if for each $QC$-function $f:X\times Y\to Z$ defined on the product of a topological space $X$ and a second-countable space $Y$ the set $D(f)$ of discontinuity points of $f$ is contained in $M\times Y$ for some meager set $M\subset X$.

\begin{theorem}\label{t:QC-stable}
The class of $QC$-Maslyuchenko spaces is $M$-stable.
\end{theorem}

\begin{proof} By Theorem~\ref{t:QC} of Bouziad and Troallic, the class of $QC$-Maslyuchenko spaces contains all metrizable spaces. It remains to prove that this class is closed under taking subspaces, homeomorphic images, countable Tychonoff products, fibered preimages, strong $\bar\sigma$-sums, $\bar G_\delta$-retral unions, Michael modifications, and sequentially admissible function spaces.
The definition of a $QC$-Maslyuchenko space ensures that the class of $QC$-Maslyuchenko spaces is closed under taking subspaces and homeomorphic images. The remaining properties of this class are less trivial and are proved in the following six lemmas.
\end{proof}

\begin{lemma} The class of $QC$-Maslyuchenko spaces is closed under taking countable Tychonoff products.
\end{lemma}

\begin{proof} Given a sequence $(Z_n)_{n\in\w}$ of $QC$-Maslyuchenko spaces, consider their Tychonoff product $Z=\prod_{n\in\w}Z_n$, and take any $QC$-function $f:X\times Y\to Z$ defined on the product of a topological space $X$ and a second-countable space $Y$. For every $n\in\w$ consider the coordinate projection $\pr_n:Z\to Z_n$. It follows that the composition $\pr_n\circ f:X\times Y\to Z_n$ is
 a $QC$-function and hence $D(\pr_n\circ f)\subset M_n\times Y$ for some meager set $M_n\subset X$. Since $D(f)=\bigcup_{n\in\w}D(\pr_n\circ f)\subset\big(\bigcup_{n\in\w}M_n\big)\times Y$, the set $D(f)$ has meager projection on $X$, witnessing that the space $Z$ is $QC$-Maslyuchenko.
\end{proof}

\begin{lemma} The class of $QC$-Maslyuchenko spaces is closed under taking fibered preimages.
\end{lemma}

\begin{proof} Fix a topological space $Z$ admitting a continuous map $g:Z\to P$ into a paracompact $QC$-Maslyuchenko space $P$ such that every point $y\in P$ has an open neighborhood $O_y\subset P$ whose preimage $g^{-1}(O_y)$ is $QC$-Maslyuchenko. Since $P$ is paracompact, the open cover $\{O_y:y\in P\}$ of $P$ can be refined by an open $\sigma$-discrete cover $\U$ such that for each set $U\in\U$ the closure $\bar U$ is contained in some set $O_{y(U)}$ and hence the space $g^{-1}(\bar U)$ is $QC$-Maslyuchenko.

To prove that the space $Z$ is $QC$-Maslyuchenko, fix a $QC$-function $f:X\times Y\to Z$ defined on the product of a topological space $X$ and second-countable space $Y$.
The continuity of $g$ implies that the composition $g\circ f:X\times Y\to P$ is a $QC$-function. Since the space $P$ is $QC$-Maslyuchenko, there is a meager subset $M_0\subset X$ such that $D(g\circ f)\subset M_0\times Y$. Let $X_0=X\setminus M_0$.

Fix a countable base $\mathcal B$ of the topology of the space $Y$. For every basic set $B\in\mathcal B$ and every $U\in\U$ consider the set $X_{B,U}=\{x\in X:f(\{x\}\times B)\subset g^{-1}(\bar U)\}$ and let $X^\circ_{B,U}$ be the interior of $X_{B,U}$ in $X$. Consider the restriction $f_{B,U}=f|X^\circ_{B,U}\times B$ of the map $f$ and observe that $f_{B,U}:X_{B,U}^\circ\times B\to g^{-1}(\bar U)$ is a $QC$-function. Since $g^{-1}(\bar U)$ is a $QC$-Maslyuchenko space, there exists a meager set $M_{B,U}\subset X_{B,U}^\circ$ such that $D(f_{B,U})\subset M_{B,U}\times B$.

The continuity of the function $f|X_0\times Y$ implies that for every $B\in\mathcal B$ the family $(X_0\cap X_{B,U})_{U\in\U}$ is $\sigma$-discrete in $X_0$ and hence the union $M_B=\bigcup_{U\in\U}M_{B,U}$ is meager in $X$. Then the union $M=\bigcup_{B\in\mathcal B}M_B$ is meager in $X$ too.

We claim that $D(f)\subset (M_0\cup M)\times Y$. Indeed, the continuity of the function $g\circ f$ at each point of the set $X_0\times Y$ guarantees that $X_0\times Y\subset \bigcup_{B\in\mathcal B}\bigcup_{U\in U}X^\circ_{B,U}\times B$. Then $$D(f)\cap (X_0\times Y)\subset \bigcup_{B\in\mathcal B}\bigcup_{U\in\U}D(f_{B,U})\subset \bigcup_{B\in\mathcal B}\bigcup_{U\in\U}(M_{B,U}\times B)\subset M\times Y$$ and hence $D(f)\subset (M_0\times Y)\cup (M\times Y)$.
\end{proof}

\begin{lemma} The class of $QC$-Maslyuchenko spaces is closed under strong $\bar\sigma$-sums.
\end{lemma}

\begin{proof} Assume that a topological space $Z$ can be written as the union $Z=\bigcup_{n\in\w}Z_n$ of an increasing sequence $(Z_n)_{n\in\w}$ of closed $QC$-Maslyuchenko subspaces of $Z$ such that each convergent sequence in $Z$ is contained in some set $Z_n$, $n\in\w$. To prove that the space $Z$ is $QC$-Maslyuchenko, fix a $QC$-function $f:X\times Y\to Z$ defined on the product of a topological space $X$ and a second-countable space $Y$. Fix a countable base $\mathcal B$ of the topology of the space $Y$. For every $n\in\w$ and $B\in\mathcal B$ consider the set
$X_{n,B}=\{x\in X:f(\{x\}\times B)\subset Z_n\}$ and let $\partial X_{n,B}$ be its boundary in $X$. We claim that $\partial X_{n,B}$ is nowhere dense in $X$.
This will follow as soon as we check that an open subset $U\subset X$ is contained in $X_{n,B}$ whenever $U\cap X_{n,B}$ is dense in $U$.
  Assuming that $U\not\subset X_{n,B}$, find a point $x\in U\setminus X_{n,B}$. Then $f(\{x\}\times B)\not\subset Z_n$ and hence $f(x,b)\in Z\setminus Z_n$ for some $b\in B$. Since $f$ is lower quasicontinuous with respect to the first variable, the set $U$ contains a non-empty open set $V$ such that $f(\{v\}\times B)\cap (Z\setminus Z_n)\ne\emptyset$ for every $v\in V$. Since $U\cap X_{n,B}$ is dense in $U$, there is a point $v\in V\cap X_{n,B}$. For this point we get $f(\{v\}\times B)\not\subset Z_n$, which contradicts the definition of the set $X_{n,B}\ni v$. This contradiction shows that $U\subset X_{n,B}$. So, the boundary $\partial X_{n,B}$ of the set $X_{n,B}$ is nowhere dense in $X$ and the interior $X_{n,B}^\circ$ of $X_{n,B}$ is dense in $X_{n,B}$.

Consider the restriction $f_{n,B}=f|X_{n,B}^\circ\times B$ and observe that $f_{n,B}:X_{n,B}^\circ\times B\to Z_n$ is a $QC$-function. Since $Z_n$ is a $QC$-Maslyuchenko space, the set $D(f_{n,B})$ is contained in $M_{n,B}\times B$ for some meager subset $M_{n,B}$ of $X_{n,B}^\circ$.
Then the set  $M=\bigcup_{n\in\w}\bigcup_{B\in\mathcal B}(\partial X_{n,B}\cup M_{n,B})$ is meager in $X$. We claim that $D(f)\subset M\times Y$. Given any point $x\in X\setminus M$, we should prove that the function $f$ is continuous at each point of the set $\{x\}\times Y$. Taking into account that the function $f^x:Y\to Z$ is continuous, $Y$ is first-countable and each convergent sequence in $Z$ is contained in some set $Z_n$, we can prove that every point $y\in Y$ has a neighborhood $B\in\mathcal B$ such that $f^x(B)\subset Z_n$ for some $n\in\w$. Then $x\in X_{n,B}$. Since $x\notin \partial X_{n,B}$, the point $x$ belongs to the interior $X_{n,B}^\circ$ of $X_{n,B}$. Since $x\notin M_{n,B}$ and $y\in B$, the function $f_{n,B}=f|X_{n,B}^\circ\times B$ is continuous at $(x,y)$ and so is the function $f$.
\end{proof}

\begin{lemma}\label{l:QC-retral} The class of $QC$-Maslyuchenko spaces is closed under $\bar G_\delta$-retral unions.
\end{lemma}

\begin{proof} We need to prove that a topological space $Z$ is $QC$-Maslyuchenko provided $Z$ is the finite union $Z=\bigcup_{\alpha\in A} Z_\alpha$ of $QC$-Maslyuchenko $\bar G_\delta$-subsets $Z_\alpha\subset Z$, which are neighborhood retracts in $Z$. For every $\alpha\in A$ fix a retraction $r_\alpha:\bar U_\alpha\to Z_\alpha$ defined on the closure $\bar U_\alpha$ of some open neighborhood $U_\alpha$ of $Z_\alpha$ in $Z$. For every $\alpha\in A$ choose a decreasing sequence $(U_{\alpha,n})_{n\in\w}$ of open sets in $Z$ such that  $Z_\alpha=\bigcap_{n\in\w}\bar U_{\alpha,n}$.

Fix a countable base $\mathcal B$ of the topology of the second-countable space $Y$. For every subset $F\subset A$, basic set $B\in\mathcal B$, and a number $n\in\w$ consider the set
$X_{F,B,n}=\{x\in X:f(\{x\}\times B)\subset \bigcap_{\alpha\in F}\bar U_\alpha\setminus \bigcup_{\beta\in A\setminus F}U_{\beta,n}\}$ and let $\partial X_{F,B,n}$ be its boundary in $X$. Repeating the argument from the preceding lemma, we can prove that the boundary $\partial X_{F,B,n}$ of $X_{F,B,n}$ is nowhere dense in $X$ and the interior $X_{F,B,n}^\circ$ of $X_{F,B,n}$ is dense in $X_{F,B,n}$.

For every $F\subset A$, $\alpha\in F$, $B\in\mathcal B$, and $n\in\w$ consider the function $f_{\alpha,F,B,n}=r_\alpha\circ f|X_{F,B,n}^\circ\times B$ and observe that $f_{\alpha,F,B,n}:X_{F,B,n}^\circ\times B\to Z_\alpha$ is a $QC$-function. Since $Z_\alpha$ is a $QC$-Maslyuchenko space, the set $D(f_{\alpha,F,B,n})\subset X_{F,B,n}^\circ\times B$ has meager projection $M_{\alpha,F,B,n}$ on $X_{F,B,n}^\circ$.

Then the set $M=\bigcup_{F\subset A}\bigcup_{\alpha\in A}\bigcup_{B\in\mathcal B}\bigcup_{n\in\w}(\partial X_{F,B,n}\cup M_{\alpha,F,B,n})$ is meager in $X$.
We claim that $D(f)\subset M\times Y$. Given any point $x\in X\setminus M$, we should prove that the function $f$ is continuous at each point of the set $\{x\}\times Y$. Fix any point $y\in Y$ and a neighborhood $O_{f(x,y)}\subset Z$ of $f(x,y)$.

Consider the non-empty subfamily $F=\{\alpha\in A:f(x,y)\in Z_\alpha\}\subset A$. Taking into account that $f(x,y)\notin\bigcup_{\beta\in A\setminus F}Z_\beta=\bigcap_{n\in\w}\bigcup_{\beta\in A\setminus F}\bar U_{\beta,n}$, find a number $n\in\w$ such that $f(x,y)\notin\bigcup_{\beta\in A\setminus F}\bar U_{\beta,n}$. The continuity of the function $f^x:Y\to Z$ at the point $y$ yields a basic neighborhood $B\in\mathcal B$ such that $f^x(B)\subset \bigcap_{\alpha\in F}U_\alpha\setminus\bigcup_{\beta\in A\setminus F}\bar U_{\beta,n}$. Then $x\in X_{F,B,n}$. Taking into account that $x\notin M$, we conclude that $x\in X_{F,B,n}\setminus \partial X_{F,B,n}=X_{F,B,n}^\circ$. Since $x\notin\bigcup_{\alpha\in F}M_{\alpha,F,B,n}$, for every $\alpha\in F$ the function $f_{\alpha,F,B,n}$ is continuous at $(x,y)$. For every $\alpha\in F$ the definition of the set $F$ ensures that $f(x,y)\in Z_\alpha$ and hence $f_{\alpha,F,B,n}(x,y)=r_\alpha\circ f(x,y)=f(x,y)\in O_{f(x,y)}$. Using the continuity of the functions $f_{\alpha,F,B,n}$, $\alpha\in F$, at $(x,y)$, we can find a neighborhood $O_{(x,y)}\subset X_{F,B,n}^\circ\times B$ of $(x,y)$ such that $f_\alpha(O_{x,y})\subset O_{f(x,y)}$ for all $\alpha\in F$. We claim that $f(O_{(x,y)})\subset O_{f(x,y)}$. Given any pair $(x',y')\in O_{x,y}\subset X_{F,B,n}\times B$, find $\alpha\in A$ such that $f(x',y')\in Z_\alpha$. It follows from $f(x',y')\in f(X_{F,B,n}\times B)\subset Z\setminus \bigcup_{\beta\in A\setminus F}Z_\beta$ that $\alpha\in F$ and hence $f(x,y)=r_\alpha\circ f(x,y)=f_{\alpha,F,B,n}(x,y)\in f_{\alpha,F,B,n}(O_{(x,y)})\subset O_{f(x,y)}$, witnessing that $f$ is continuous at $(x,y)$.
\end{proof}

\begin{lemma} The class of $QC$-Maslyuchenko spaces is closed under taking Michael modifications.
\end{lemma}

\begin{proof} We need to prove that for any map $g:L\to Z$ from a locally compact paracompact $QC$-Maslyuchenko space $L$ to a $QC$-Maslyuchenko space $Z$ the Michael modification $Z_g$ is $QC$-Maslyuchenko. Fix a $QC$-function $f:X\times Y\to Z_g$ defined on the product of a topological space $X$ and a second countable space $Y$. The definition of the topology on the space $Z_g$ guarantees that the map $\pr:Z_g\to Z$ defined by $\pr|L=g$ and $\pr|Z\setminus g(L)=\id$ is continuous. Then the composition $\pr\circ f:X\times Y\to Z$ is a $QC$-function. Since the space $Z$ is $QC$-Maslyuchenko, there is a meager set $M_0\subset X$ such that $D(\pr\circ f)\subset M_0\times Y$.

Fix a countable base $\mathcal B$ of the topology of the space $Y$. By the paracompactness, the locally compact space $L$ admits a locally finite open cover $\U$ such that each set $U\in\U$ has compact closure $\bar U$ in $L$. Let $\bar\U=\{\bar U:U\in\U\}$. For any subset $A\subset L$ consider its $\U$-star $\St(A,\U)=\bigcup\{U\in\U:A\cap U\ne\emptyset\}$ and observe that for a compact set $A\subset L$ the star $\St(A,\bar \U)$ is compact.

For every $B\in\mathcal B$ consider the subset
$X_{B}=\{x\in X:\exists U\in\U\;f(\{x\}\times B)\subset U\}$ of $X$. Here we identify $L$ with a subspace of $Z_g$. Let $\bar X_{B}^\circ$ be the interior of the closure $\bar X_{B}$ of $X_B$ in $X$. We claim that for each point $x\in \bar X^{\circ}_{B}$ and each $b\in B$ we get
$f(\{x\}\times B)\subset \St(\St(f(x,b),\bar\U),\bar \U)\subset L$. Assuming the converse, we could find a point $y\in B$ such that $f(x,y)\notin \St(\St(f(x,b),\bar\U),\bar \U)$.
The compactness of the set $S=\St(\St(f(x,b),\bar\U),\bar \U)$ guarantees that its complement $Z_g\setminus S$ is an open neighborhood of $f(x,y)$ in $Z_g$.
Then we can apply the  lower quasicontinuity of $f$ by the first variable, and find a non-empty open set $W\subset \bar X_{B}^\circ$ such that
$f(\{w\}\times B)\cap \St(f(x,b),\U)\ne\emptyset\ne f(\{w\}\times B)\cap (Z_g\setminus S)$ for every $w\in W$. Then for any point $w\in W\cap X_{B}$ we get $f(\{w\}\times B)\subset \St(\St(f(x,b),\U),\U)\subset S$, which contradicts the choice of $W$. This contradiction shows that $f(\{x\}\times B)\subset S\subset L$ for every $x\in \bar X_B^\circ$. Now consider the restriction $f_B=f|\bar X_B^\circ\times B$. Since the space $L$ is $QC$-Maslyuchenko, for the $QC$-function $f_B:\bar X^\circ_B\times B\to L$ there exists a meager set $M_B\subset \bar X_B^\circ$ such that $D(f_B)\subset M_B\times B$. Consider the meager set $M=M_0\cup\bigcup_{B\in\mathcal B}(\partial \bar X_{B}\cup M_B)$ in $X$. We claim that $D(f)\subset M\times Y$. Given any point $x\in X\setminus M$ and $y\in Y$, we should prove that $f$ is continuous at $(x,y)$. If $f(x,y)\in L$, then by the continuity of the function $f^x:Y\to Z_g$, we can find a basic neighborhood $B\in\mathcal B$ of $y$ such that $f(\{x\}\times B)\subset U$ for some $U\in\U$. Then $x\in X_B\subset \bar X_B$. Since $x\notin \partial \bar X_B$, the point $x$ belongs to the open set $\bar X_B^\circ$. It follows from $x\notin M$ that $x\notin M_B$ and hence the function $f_B=f|\bar X_B^\circ\times B$ is continuous at $(x,y)$ and so is the function $f$.

If $f(x,y)\notin L$, then the continuity of the function $f$ at $(x,y)$ follows from the continuity of the function $\pr\circ g:X\times Y\to Z$ at $(x,y)$ and the definition of the topology on the space $Z_g$ (the function $\pr\circ g$ is continuous at $(x,y)$ as $x\notin M_0$).
\end{proof}

\begin{lemma} The class of $QC$-Maslyuchenko spaces is closed under taking sequentially admissible function spaces.
\end{lemma}

\begin{proof} Let $X$ be an $\aleph_0$-space, $Y$ be a $QC$-Maslyuchenko space, and $Z\subset C(X,Y)$ be a subspace endowed with a sequentially admissible topology $\tau$.
 To prove that the topological space $(Z,\tau)$ is $QC$-Maslyuchenko, take any $QC$-function $f:T\times S\to Z$ defined on the product of a topological space $T$ and a second-countable space $S$. To each point $x\in X$ assign the $Y$-valued Dirac measure $\delta_x:C(X,Y)\to Y$, $\delta_x:\varphi\mapsto \varphi(x)$. Since the topology $\tau$ of $Z$ is admissible, the restriction $\delta_x|Z:Z\to Y$ is continuous. This implies that for any $x\in X$ the composition $\delta_x\circ f:T\times S\to Y$ remains a $QC$-function.

By \cite{Mich} (see also Remark in \cite[p.494]{Grue}), the $\aleph_0$-space $X$ is the image of a metrizable separable space $R$ under a continuous compact-covering map $\xi:R\to X$. The compact-covering property of $\xi$ implies that the dual map $C_k\xi:C_k(X,Y)\to C_k(R,Y)$, $C_k\xi:\varphi\mapsto \varphi\circ\xi$, is a topological embedding.

  Consider the function $g:T\times S\times R\to Y$ assigning to each triple $(t,s,r)\in T\times S\times R$ the point $\delta_{\xi(r)}\big(f(t,s)\big)\in Y$. We claim that $g:T\times (S\times R)\to Y$ is a $QC$-function.

To see that $g$ is lower quasicontinuous with respect to the first variable, fix a triple
$(t,s,r)\in T\times S\times R$ and open neighborhoods $O_t\subset T$, $O_s\subset S$, $O_r\subset R$, and $O_{g(t,s,r)}\subset Z$ of the points $t,s,r$ and $g(t,s,r)$, respectively.
  Since $g(t,s,r)=\delta_{\xi(r)}(f(t,s))$ and $\delta_{\xi(r)}\circ f:T\times S\to Y$ is a $QC$-function, there exists a non-empty open set $U\subset O_t$ such that $\emptyset\ne \delta_{\xi(r)}\circ f(\{u\}\times O_s)\cap O_{g(t,s,r)}=g(\{u\}\times O_s\times\{r\})\cap O_{g(t,s,r)}\subset g(\{u\}\times O_s\times O_r)\cap O_{g(t,s,r)}$ for all $u\in U$, witnessing that the function $g$ is lower quasicontinuous with respect to the first variable.

 Next, we show that $g$ is continuous with respect to the second variable, which means that for any $t\in T$ the function $g^t:S\times R\to Y$, $g^t:(s,r)\mapsto g(t,s,r)$, is continuous. Since $f:T\times S\to Z$ is a $QC$-function, the function $f^t:S\to Z$, $f^t:s\mapsto f(t,s)$, is continuous. Since the topology on $Z$ is sequentially admissible, the evaluation map $e:Z\times X\to Y$, $e:(z,x)\mapsto z(x)=\delta_x(z)$, is sequentially continuous. Observing that $g^t(s,r)=\delta_{\xi(r)}(f(t,s))=e(f^t(s),\xi(r))$, we see that the function $g^t$ is sequentially continuous and hence continuous as the space $S\times R$ is first countable.

  Therefore, $g:T\times (S\times R)\to Y$ is a $QC$-function. Since $Y$ is a $QC$-Maslyuchenko space, the set $D(g)\subset T\times (S\times R)$ has meager projection $M$ on $T$. It remains to prove that $D(f)\subset M\times S$. Given any point $(t,s)\in (T\setminus M)\times S$ we should check that $f$ is continuous at $(t,s)$. As $g$ is a $QC$-function, the function $g^{(t,s)}:R\to Y$, $g^{(t,s)}:r\mapsto g(t,s,r)$, is continuous.

    The continuity of the function $g:T\times S\times R\to Y$ at each point of the set $(X\setminus M)\times S\times R$ implies that the map $\dot g:T\times S\to C_k(R,Y)$, $\dot g:(t,s)\mapsto g^{(t,s)}$, is continuous at each point of the set $(T\setminus M)\times S$. Observe that for any triple $(t,s,r)\in T\times S\times R$  we get $g^{(t,s)}(r)=\delta_{\xi(r)}(f(t,s))$ and hence $g^{(t,s)}=C_k\xi(f(t,s))\in C_k\xi(C_k(X,Y))$.
Since  $C_k\xi:C_k(X,Y)\to C_k(R,Y)$ is a topological embedding, the map $f:T\times S\to (Z,\tau_k|Z)\subset C_k(X,Y)$ is continuous at each point of the set $(T\setminus M)\times S$. Since the topology on $Z$ is $k$-admissible, the identity map $(Z,\tau_k|Z)\to Z$ is continuous, which implies that the map $f:T\times S\to Z$ is continuous at each point of the set $(T\setminus M)\times S$. This completes the proof of the inclusion $D(f)\subset M\times S$.
\end{proof}

Theorem~\ref{t:QC-stable} and Proposition~\ref{p:star} imply:

\begin{corollary} A paracompact space $X$ is $QC$-Maslyuchenko if and only if $X$ has a star-countable cover by open $QC$-Maslyuchenko subspaces.
\end{corollary}

Taking into account that the class of $QC$-Maslyuchenko spaces contains all metrizable spaces and is closed under strong $\bar\sigma$-sums, we obtain the following generalization of a result of Maslyuchenko, Mykhaylyuk and Shishina \cite{MMS} (who proved that for every strongly $\sigma$-metrizable space $Z$ and separately continuous function $f:X\times Y\to Z$ defined on the product of a topological space $X$ and a compact metrizable space $Y$ the set $D(f)$ has meager projection on $X$).

\begin{corollary} Each strongly $\sigma$-metrizable space is $QC$-Maslyuchenko.
\end{corollary}

Theorems~\ref{t:QC-stable} and \ref{t:aleph0} imply:

\begin{corollary}\label{c:QC-aleph0} Each $\aleph_0$-space is $QC$-Maslyuchenko.
\end{corollary}

Taking into account that the class of $QC$-Maslyuchenko spaces is closed under Michael modifications, we get the following example showing that the class of $QC$-Maslyuchenko spaces is not contained in the class of $\sigma$-spaces.

\begin{example} The Michael line $\IR_{\IQ}$ is a $QC$-Maslyuchenko space which fails to be a $\sigma$-space.
\end{example}

\section{$KC$-Maslyuchenko spaces}\label{s:KC}

We recall that a topological space $Z$ is called {\em $KC$-Maslyuchenko} if for each $KC$-function $f:X\times Y\to Z$ defined on the product of a topological space $X$ and a second-countable space $Y$ the set $D(f)$ of discontinuity points of $f$ has meager projection on $X$ and hence is contained in the product $M\times Y$ for some meager set $M\subset X$.

Modifying the proof of Theorem~\ref{t:QC-stable} one can prove:

\begin{theorem}\label{t:KC-stable} The class of $KC$-Maslyuchenko spaces is $M$-stable.
\end{theorem}

This theorem combined with Proposition~\ref{p:star} implies:

\begin{corollary} A paracompact space $X$ is $KC$-Maslyuchenko if and only if $X$ has a star-countable cover by open $KC$-Maslyuchenko subspaces.
\end{corollary}

Since each $KC$-function is a $QC$-function, each $QC$-Maslyuchenko space is $KC$-Maslyuchenko.

\begin{problem} Is there a $KC$-Maslyuchenko space which is not $QC$-Maslyuchenko?
\end{problem}

Theorems~\ref{t:KC-stable} and \ref{t:aleph0} imply that the classes of $QC$-Maslyuchenko and $KC$-Maslyuchenko spaces contain all metrizable spaces and all $\aleph_0$-spaces. The class of (paracompact) $\aleph$-spaces also contains all metrizable spaces and all $\aleph_0$-spaces. This suggests the following open problem.

\begin{problem} Is each (paracompact) $\aleph$-space $KC$-Maslyuchenko?  $QC$-Maslyuchenko?
\end{problem}

The following proposition gives a partial answer to this problem. Let us recall that a topological space $X$ is called {\em countably cellular} if any disjoint family of non-empty open sets in $X$ is at most countable.

\begin{proposition} Let $Z$ be a paracompact $\aleph$-space. For any $KC$-function $f:X\times Y\to Z$ defined on the product of a countably cellular topological space $X$ and a second-countable space $Y$ the set $D(f)$ has meager projection on $X$.
\end{proposition}

\begin{proof} Let $W$ be the union of all open meager subspaces of $X$ and observe that $W$ is meager in $X$. If $W=X$, then the space $X$ is meager and hence the projection of the set $D(f)$ on $X$ is meager. So, we assume that $W\ne X$. In this case the subspace $X\setminus W$ has non-empty interior $X^\bullet$ in $X$ and $X^\bullet$ is a Baire space. By Theorem~3 in \cite{MN00}, any $KC$-function defined on the product of a Baire space and a first-countable space  with values in a Tychonoff space is quasicontinuous. Consequently the restriction $f|X^\bullet\times Y$ is quasicontinuous. The countable cellularity of $X$ implies the countably cellularity of the Baire space $X^\bullet$. Using the second-countability of $Y$, we can show that the product $X^\bullet\times Y$ is countably cellular. The quasicontinuity of the function $f|X^\bullet\times Y$ implies that the countably cellular space $X^\bullet\times Y$ has countably cellular image $f(X^\bullet\times Y)$. By Lemma~\ref{l:cellular} proved below, the space $f(X^\bullet\times Y)$ has countable $k$-network and hence is an $\aleph_0$-space. It follows that the restriction $f|X^\bullet\times Y:X^\bullet\times Y\to f(X^\bullet\times Y)\subset Z$ is a $KC$-functions (and hence a $QC$-function) with values in an $\aleph_0$-space. By Corollary~\ref{c:QC-aleph0}, the set $D(f|X^\bullet\times Y)$ has meager projection on $X^\bullet$. Since the set $X\setminus X^\bullet$ is meager in $X$, the projection of the set $D(f)$ on $X$  is meager in $X$.
\end{proof}

\begin{lemma}\label{l:cellular} If $\mathcal N$ is a $\sigma$-locally finite family of subsets of a paracompact space $X$, then for any countably cellular subspace $C\subset X$ the subfamily $\{N\in\mathcal N:C\cap N\ne\emptyset\}$ is at most countable.
\end{lemma}

\begin{proof} Write the family $\mathcal N$ as the countable union $\mathcal N=\bigcup_{i\in\w}\mathcal N_i$ of locally finite families $\mathcal N_i$, $i\in\w$.
For every $i\in\w$ we can find an open cover $\U_i$ of the space $X$ such that for each open set $U\in\U_i$ the family $\{N\in\mathcal N_i:U\cap N\ne\emptyset\}$ is finite. Using the paracompactness of $X$, choose a $\sigma$-discrete open cover $\V_i$ refining the cover $\U_i$ (see \cite[5.1.16]{En}). Since the space $C$ is countably cellular, the subfamily $\V'_i=\{V\in\V_i:C\cap V\ne\emptyset\}$ is at most countable. Since each set $V\in\V_i$ is contained in some set $U\in\U_i$, the set $\mathcal N_i'=\bigcup_{V\in\V_i'}\{N\in\mathcal N_i:N\cap V\ne\emptyset\}$ is at most countable too. Then the family $\mathcal N'=\bigcup_{i\in\w}\mathcal N_i'$ is at most countable and contains the family $\{N\in\mathcal N:N\cap C\ne\emptyset\}$. \end{proof}

The $M$-stability of the class of $KC$-Maslyuchenko spaces implies that for any $\aleph_0$-space $X$ and any $KC$-Maslyuchenko space $Y$ the function space $C_k(X,Y)$ is $KC$-Maslyuchenko. For cosmic spaces $X$ we can prove a bit weaker result.

\begin{theorem}\label{t:K-func} If $Y$ is a $KC$-Maslyuchenko space, then for any cosmic space $X$, any space $Z\subset C(X,Y)$ endowed with an admissible topology $\tau$ is Piotrowski.
\end{theorem}

\begin{proof} To show that the topological space $Z$ is Piotrowski, fix any quasicontinuous function $f:T\to Z$ defined on a topological space $T$. Consider the evaluation function $e:Z\times X\to Y$, $e:(z,x)\mapsto z(x)$, which is separately continuous because of the inclusion $\tau|Z\subset\tau_p$. Then the function $\ddot f:T\times X\to Y$, $\ddot f:(t,x)\mapsto e(f(t),x)$, is a $KC$-function. Since the space $Y$ is $KC$-Maslyuchenko, there exists a meager set $M\subset T$ such that $D(\ddot f)\subset M\times X$. We claim that $D(f)\subset M$. We should prove that the function $f$ is continuous at each point $t\in T\setminus M$. Using the continuity of the function $\ddot f:T\times X\to Y$ at each point of the set $\{t\}\times X$, it can be shown that the function $f:T\to (Z,\tau_k)\subset C_k(X,Y)$ is continuous at $t$. Since $\tau_k|Z\subset\tau$, the function $f:T\to (Z,\tau)$ is continuous at $t$ too. So, $D(f)\subset M$.
\end{proof}

Applying Piotrowski spaces to $KC$-functions we obtain the following useful fact, which can be compared with Theorem 11 in \cite{KKM}.

\begin{theorem} Let $Z$ be a Piotrowski Tychonoff space. For each $KC$-function $f:X\times Y\to Z$ defined on the product of a topological space $X$ and a first-countable space $Y$ the set $D(f)$ of discontinuity points of $f$ is meager in $X\times Y$.
\end{theorem}

\begin{proof} Let $W$ be the union of all open meager subspaces of $X$. If $W=X$, then the spaces $X$ and $X\times Y$ are meager an so is the set $D(f)$ in $X\times Y$. So, we assume that $W\ne X$. In this case the complement $X\setminus W$ has non-meager interior $X^\bullet$ in $X$ and the space $X^\bullet$ is Baire. We claim that the function $f|X^\bullet\times Y$ is quasicontinuous. Fix any point $(x,y)\in X^\bullet\times Y$ and open sets $O_{x,y}\subset X\times Y$, $O_{f(x,y)}\subset Z$ such that $(x,y)\in O_{x,y}$ and $f(x,y)\in O_{f(x,y)}$. Since $Z$ is Tychonoff, there exists a continuous function $g:Z\to[0,1]$ such that $f(x,y)\in g^{-1}(0)\subset g^{-1}([0,1))\subset O_{f(x,y)}$. By Theorem 4.4 of \cite{BT}, the real-valued $KC$-function $g\circ f|X^\bullet\times Y:X^\bullet\times Y\to [0,1]$ is quasi-continuous (more precisely, $D(f)\cap (X^\bullet\times\{y\})$ is meager in $X^\bullet\times\{y\}$). Then we can find a non-empty open subset $U\subset O_{x,y}$ such that $g\circ f(U)\subset[0,1)$ and hence $f(U)\subset g^{-1}([0,1))\subset O_{f(x,y)}$, witnessing that the function $f|X^\bullet\times Y$ is quasi-continuous. Taking into account that $X$ is a regular Piotrowski space, we conclude that the set $D(f|X^\bullet\times Y)$ is meager in $X^\bullet \times Y$ (see \cite{BM}). Since the set $X\setminus X^\bullet=\overline{W}$ is meager in $X$, the set $D(f)\subset \overline{W}\times Y\cup D(f|X^\bullet\times Y)$ is meager in $X\times Y$.
\end{proof}

\section{$CC$-Maslyuchenko spaces}\label{s:CC}

In this section we study $CC$-Maslyuchenko spaces. We recall that a topological space $Z$ is {\em $CC$-Maslyuchenko} if for any separately continuous function $f:X\times Y\to Z$ defined on the product of separable metrizable spaces $X,Y$ the set $D(f)$ has meager projections on $X$ and $Y$.

Modifying the proof of Theorems~\ref{t:QC-stable}, we can prove

\begin{theorem}\label{t:CC-stable} The class of $CC$-Maslyuchenko spaces is $M$-stable.
\end{theorem}

This theorem combined with Proposition~\ref{p:star} implies:

\begin{corollary} A paracompact space $X$ is $CC$-Maslyuchenko if and only if $X$ has a star-countable cover by open $CC$-Maslyuchenko subspaces.
\end{corollary}

The following proposition shows that the class of $CC$-Maslyuchenko spaces is separably determined.

\begin{proposition}\label{p:separable} A topological space $Z$ is $CC$-Maslyuchenko if and only if for each countable subset $A\subset Z$ its second sequential closure $\cl_1(\cl_1(A))$ is  $CC$-Maslyuchenko.
\end{proposition}

\begin{proof} The ``only if'' part is trivial. To prove the ``if'' part, assume that for each countable subset $A\subset Z$ the second separable closure $\cl_1(\cl_1(A))$ of $A$ in $Z$ is $CC$-Maslyuchenko.
To prove that the space $Z$ is $CC$-Maslyuchenko, fix a separately continuous function $f:X\times Y\to Z$ defined on the product of two separable metrizable spaces. Fix countable dense sets $D_X\subset X$ and $D_Y\subset Y$ and consider the countable dense subset $D=D_X\times D_Y$ of the product $X\times Y$. Using the separate continuity of $f$, it can be shown that for the countable set $A=f(D)$ we get $f(X\times Y)\subset\cl_1(\cl_1(A))\subset Z$. By our assumption, the space $\cl_1(\cl_1(A))$ is $CC$-Maslyuchenko, which implies that the set $D(f)$ of discontinuity points of $f$ has meager projections on $X$ and $Y$.
\end{proof}

Corollary~\ref{c:QC-aleph0} and Proposition~\ref{p:separable} imply:

\begin{corollary} Each $\aleph_0$-monolithic space is $CC$-Maslyuchenko. Consequently, the class of $CC$-Maslyuchenko spaces contains all $k^*$-metrizable spaces and  all $\aleph$-spaces.
\end{corollary}

We recall that a topological space $X$ is {\em $\aleph_0$-monolithic} if each countable subspace $A\subset X$ is an $\aleph_0$-space and for each $\aleph_0$-subspace $A\subset X$ the sequential closure $\cl_1(A)$ is an $\aleph_0$-space. Corollary~{5.12} in \cite{BBK} implies that each $k^*$-metrizable space is $\aleph_0$-monolithic. By Corollary~7.5 \cite{BBK}, each $\aleph$-space is $k^*$-metrizable.
Therefore, we get the implications:
$$\mbox{$\aleph_0$-space $\Ra$ $\aleph$-space $\Ra$ $k^*$-metrizable $\Ra$ $\aleph_0$-monolithic $\Ra$ $CC$-Maslyuchenko.}
$$

By Theorem~\ref{t:CC-stable}, the countable product of $CC$-Maslyuchenko spaces is $CC$-Maslyuchenko. In fact, we can prove a bit more.

  By $\add(\mathcal M)$ we denote the smallest cardinality of a family $\A$ of meager subsets on the real line whose union $\bigcup\A$ is not meager in $\IR$. It is well-known (\cite{Vau}, \cite{Blass})  that $\w_1\le\add(\mathcal M)\le\mathfrak c$, and $\add(\M)=\mathfrak c$ under Martin's Axiom.

\begin{proposition}\label{p:CCprod} For any family $(Z_\alpha)_{\alpha\in A}$ of $CC$-Maslyuchenko spaces with $|A|<\add(\M)$,  the Tychonoff product $Z=\prod_{\alpha\in A}Z_\alpha$ is $CC$-Maslyuchenko.
\end{proposition}

\begin{proof} Fix any separately continuous function $f:X\times Y\to Z$ defined on the product of separable metrizable spaces. For every $\alpha\in A$ consider the coordinate projection $\pr_\alpha:Z\to Z_\alpha$, $\pr_\alpha:(x_\beta)_{\beta\in A}\mapsto x_\alpha$. Since each space $Z_\alpha$ is $CC$-Maslyuchenko,
the set $D(\pr_\alpha\circ f)$ is contained in the product $M_\alpha\times N_\alpha$ of two meager  sets $M_\alpha\subset X$ and $N_\alpha\subset Y$. The following (known) lemma implies that the unions $M=\bigcup_{\alpha\in A}M_\alpha$ and $N=\bigcup_{\alpha\in A}N_\alpha$ are meager and hence the set $D(f)\subset M\times N$ has meager projections on $X$ and $Y$.
\end{proof}

\begin{lemma} Let $\A$ be a family of meager subsets of a metrizable separable space $X$. If $|\A|<\add(\M)$, then the union $\bigcup\A$ is meager in $X$.
\end{lemma}

\begin{proof} Replacing the space $X$ by its completion (with respect to any metric generating the topology of $X$), we can assume that the space $X$ is Polish. By Cantor-Bendixon Theorem, $X$ can be written as the disjoint union $X=C\cup F$ of an open countable subspace $C$ and a closed subspace  $F$ having no isolated point.
Let $F^\circ$ and $\partial F$ be the interior and boundary of $F$ in $X$, respectively.
By Baire Theorem, the set $\dot C$ of isolated point of the countable Polish space $C$ is dense in $C$. Each set $A\in\A$ is meager in $X$ and hence contain no isolated point of $C$. Then the set $C\cap \bigcup \A$ does not intersects the open dense subset $\dot C$ of $C$ and hence is nowhere dense in $C$ and in $X$. Since the set $\partial F$ is nowhere dense in $X$, the intersection $\bigcup\A\cap\partial F$ is nowhere dense in $X$ too. It remains to prove that the intersection  $F^\circ\cap\bigcup\A$ is meager in $F^\circ$. The Polish space $F^\circ$  contains no isolated points and hence contains a dense $G_\delta$-subset $Z$ homeomorphic to the space $\IR\setminus\IQ$ of irrational numbers (this follows from the Lavrentiev Theorem \cite[4.3.21]{En}). Now the definition of the cardinal $\add(\M)$ implies that the union $Z\cap\bigcup\A$ is meager in $Z$ and hence $F^\circ\cap\bigcup\A$ is meager in $F^\circ$ and also in $X$. This completes the proof of the lemma.
\end{proof}

Since each Tychonoff space of weight $\kappa$ embeds into the Tychonoff cube $[0,1]^\kappa$, Theorem~\ref{t:CC} and Proposition~\ref{p:CCprod} imply:

\begin{corollary}\label{c:CCweight} Each Tychonoff space $X$ of weight $w(X)<\add(\M)$ is $CC$-Maslyuchenko.
\end{corollary}

The cardinal $\add(\M)$ in Proposition~\ref{p:CCprod} and Corollary~\ref{c:CCweight} cannot be replaced by a larger cardinal.

\begin{proposition} The Tychonoff power $\IR^{\add(\M)}$ is not $CC$-Maslyuchenko.
\end{proposition}

\begin{proof} By definition of the cardinal $\add(\M)$, on the real line there is a family $\A$ consisting of $|\A|=\add(\M)$ meager sets such that the union $\bigcup\A$ is not meager in $\IR$. Replacing each $A\in\A$ by a larger meager set, we can assume that $A$ is of type $F_\sigma$ in $\IR$.
Also we can assume that the family $\A$ is closed under finite unions. By a classical result of Baire (see also \cite{MMSob}), for any meager $F_\sigma$-set $A\in\A$ there exists a separately continuous function $f_A:\IR\times\IR\to\IR$ such that $D(f)=A\times A$. Then the function $f=(f_A)_{A\in\A}:\IR\times\IR\to\IR^{\A}$, $f:(x,y)\mapsto (f_A(x,y))_{A\in\A}$, has $D(f)=(\bigcup\A)\times(\bigcup\A)$ and this set is not (projectively) meager in $\IR\times\IR$.
\end{proof}

Using Propositions~\ref{p:CCprod} and \ref{p:separable}, we shall prove that the class of Maslyuchenko spaces is closed under $\Sigma^{<\kappa}$-products with $\kappa\le\add(\M)$.
For a family $(X_\alpha,*_\alpha)$, $\alpha\in A$, of pointed topological spaces and a cardinal $\kappa$ the subspace
$$\Sigma_{\alpha\in A}^{<\kappa}(X_\alpha,*_\alpha)=\{(x_\alpha)_{\alpha\in A}\in\prod_{\alpha\in A}X_\alpha:|\{\alpha\in A:x_\alpha\ne *_\alpha\}|<\kappa\}$$
of the Tychonoff product $\prod_{\alpha\in A}X_\alpha$ will be called the {\em $\Sigma^{<\kappa}$-product} of the family of pointed spaces $X_\alpha$, $\alpha\in A$. Observe that the $\Sigma^{<\kappa}$-product consists of elements $x=(x_\alpha)_{\alpha\in A}\in\prod_{\alpha\in A}X_\alpha$ whose support $\supp(x)=\{\alpha\in A:x_\alpha\ne *_\alpha\}$ has cardinality $|\supp(x)|<\kappa$.

We shall say that a class $\C$ of topological spaces is {\em closed under taking $\Sigma^{<\kappa}$-products} if for any family of $T_1$-pointed spaces $(X_\alpha,*_\alpha)_{\alpha\in A}$ the $\Sigma^{<\kappa}$-product $\Sigma^{<\kappa}_{\alpha\in A}(X_\alpha,*_\alpha)$ belongs to the class $\C$.
We recall that a pointed space $(X,*)$ is called {\em $T_1$-pointed} if the singleton $\{*\}$ is closed in $X$.

\begin{proposition}\label{p:sigma} The class of $CC$-Maslyuchenko spaces is closed under taking $\Sigma^{<\add(\M)}$-products.
\end{proposition}

\begin{proof} Fix any $T_1$-pointed $CC$-Maslyuchenko spaces $(X_\alpha,*_\alpha)$, $\alpha\in A$. We claim that each separable subspace $S\subset\Sigma^{<\add(\M)}_{\alpha\in A}X_\alpha$ is $CC$-Maslyuchenko.
Fix a countable dense subset $D\subset S$ and consider the set $B=\bigcup_{x\in B}\supp(x)$. Since the cardinal $\add(\M)$ is regular (and hence has uncountable cofinality), the set $B$ has cardinality $|B|<\add(\M)$. It follows that
$$S\subset\bar D\subset \{x\in\prod_{\alpha\in A}X_\alpha:\supp(x)\subset B\}$$ and hence $S$ is homeomorphic to a subspace of the product $\prod_{\alpha\in B}X_\alpha$, which is $CC$-Maslyuchenko by Proposition~\ref{p:CCprod}. Consequently, the space  $S$ is $CC$-Maslyuchenko too. By Proposition~\ref{p:separable}, the space $\Sigma^{<\add(\M)}_{\alpha\in A}(X_\alpha,*_\alpha)$ is $CC$-Maslyuchenko.
\end{proof}

A compact space is called {\em Corson compact} if it is homeomorphic to a subspace of a $\Sigma^{<\w_1}$-product of real lines. Proposition~\ref{p:sigma} implies:

\begin{corollary}\label{c:CC-Corson} The class of $CC$-Maslyuchenko spaces contains all Corson compact spaces.
\end{corollary}

The one-point compactification $\alpha D=\{\infty\}\cup D$ of any uncountable discrete space $D$, being Corson (even Eberlien) compact, is $CC$-Maslyuchenko. On the other hand, this space is not $KC$-Maslyuchenko.

\begin{example}\label{e:aD} For every uncountable discrete space $D$ there is a separately continuous map $f:X\times \alpha\w\to \alpha D$ defined on the product of a complete metric space $X$ and the one-point compactification of the countable discrete space $\w$ such that $D(f)=X\times\{\infty\}$ and hence the set $D(f)$ has non-meager projection onto $X$.
\end{example}

\begin{proof} It is clear that the topology of the countable product $D^\w$ is generated by a complete metric. It can be shown that the set $X=\{x\in D^\w:f$ is injective$\}$ is closed in $D^\w$. It follows that each injective map $x:\w\to D$ admits a unique continuous extension $\bar x:\alpha\w\to \alpha D$. It can be shown that the function $f:X\times\alpha\w\to \alpha D$, $f:(x,n)\mapsto \bar x(n)$, is separately continuous. Let us check that this function is discontinuous at each point $(x,\infty)$, $x\in X$. Take any point $a\in D\setminus x(\w)$ and consider the neighborhood $U=\alpha D\setminus\{a\}$ of the compactifying point $\infty$ in $\alpha D$. Assuming that the function $f$ is continuous at $(x,\infty)$, we can find a neighborhood $O_x\subset X$ of $x$ and a number $n_0\in\w$ such that $y(n)\in U$ for any $y\in O_x$ and $n\ge n_0$. By the definition of the Tychonoff product topology on $D^\w$, there exists a finite subset $F\subset\w$ such that $\{y\in X:y|F=x|F\}\subset O_x$. Now choose an injective function $y:\w\to D$ such that $y|F=x|F$ and $y(n)=a$ for some $n\ge n_0$. Then $y(n)=a\notin U$, which contradicts the choice of $O_x$ and $n_0$.
\end{proof}

Next, we construct an example of $CC$-Maslyuchenko space which is not  Piotrowski.
Consider the $\Sigma^{<\w_1}$-product
$$\Sigma^{<\w_1}_2=\{x\in\{0,1\}^{\w_1}:|\{\alpha\in\w_1:x_\alpha\ne 0\}|\le\w\}\subset\{0,1\}^{\w_1}$$of $\w_1$-many doubletons.

\begin{example}\label{e:CCnotP} The space $\Sigma^{<\w_1}_2$ is  $CC$-Maslyuchenko but not Piotrowski.
\end{example}

 \begin{proof} By Proposition~\ref{p:sigma}, the space $\Sigma^{<\w_1}_2$ is $CC$-Maslyuchenko.
To show that $\Sigma^{<\w_1}_2$ is not Piotrowski, consider the Banach space $\ell_1(\w_1)=\{x\in\IR^{\w_1}:\sum_{\alpha\in\w_1}|x(\alpha)|<\infty\}$ endowed with the norm $\|x\|=\sum_{\alpha\in\w_1}|x(\alpha)|$. It is easy to see that each point $x\in\ell(\w_1)$ has countable support $\supp(x)=\{\alpha\in\w_1:x(\alpha)\ne0\}$. Consequently, $\ell_1(\w_1)=\bigcup_{\alpha<\w_1}\ell_1(\alpha)$ where $\ell_1(\alpha)=\{x\in\ell_1(\w_1):\supp(x)\subset[0,\alpha)\}$. It follows that for every countable ordinal $\alpha$ the characteristic function $f_\alpha:\ell_1(\w_1)\to\{0,1\}$ of the set $\ell_1(\alpha)$ is quasicontinuous. Then the function $f=(f_\alpha)_{\alpha\in\w_1}:\ell_1(\w_1)\to\Sigma_2^{<\w_1}$ is quasicontinuous too. Since $D(f)=\ell_1(\w_1)$, the space $\Sigma^{<\w_1}_2$ fails to be Piotrowski.
\end{proof}

\begin{remark} By \cite[9.12]{Tod84}, there exists a first-countable Corson compact space $X$ containing no dense metrizable subspaces. By  Corollary~\ref{c:CC-Corson}, the space $X$ is $CC$-Masluychenko. Since each compact Piotrowski space    contains a metrizable dense $G_\delta$-subset (see \cite{BM}), the Corson compact space $X$ is not Piotrowski.
\end{remark}

%\begin{remark} In Example~\ref{e:PnotCC} we shall show that the function space $C_p[0,1]$ is Piotrowski but not $CC$-Maslyuchenko.
%\end{remark}

The following example was constructed in \cite{MM12} (see also Theorem 3.3.2 \cite{Myronyk}).

\begin{example}\label{e:PnotCC} The function space $C_p[0,1]$ is cosmic and Piotrowski but not $CC$-Maslyuchenko.
\end{example}

\begin{proof} Let $sp:\IR\times\IR\to\IR$ be the standard discontinuous separately continuous function defined by the formula
$$sp(x,y)=\begin{cases}
\frac{2xy}{x^2+y^2}&\mbox{if $x^2+y^2>0$};\\
0&\mbox{otherwise}.
\end{cases}
$$For two real numbers $a,b\in[0,1]$ consider the continuous function $sp_{a,b}\in C_p[0,1]$ defined by the formula $sp_{a,b}(t)=sp(a-t,b)+sp(a,b-t)$ for $t\in[0,1]$. It can be shown (see Theorem 3.3.2 in \cite{Myronyk}) that the function $sp_{**}:[0,1]\times[0,1]\to C_p[0,1]$, $sp_{**}:(a,b)\mapsto sp_{a,b}$, is separately continuous and $D(sp_{**})=([0,1]\times\{0\})\cup(\{0\}\times[0,1])$, which implies that the space $C_p[0,1]$ is not $CC$-Maslyuchenko. Being cosmic, the space $C_p[0,1]$ is strictly fragmentable (see \cite[Proposition 2.1]{KM12} or \cite{BM}) and hence Piotrowski.
\end{proof}

\begin{example}\label{e:P+CCnotKC} The one-point compactification $\alpha D$ of an uncountable discrete space $D$ is Piotrowski and $CC$-Maslyuchenko but not $KC$-Maslyuchenko.
\end{example}

\begin{proof} The space $\alpha D$ being scattered is strictly fragmentable (by the discrete $\{0,1\}$-valued metric) and hence Piotrowski. Being Corson compact, the space $\alpha D$ is $CC$-Maslyuchenko. By Example~\ref{e:aD}, the space $\alpha D$ is not $KC$-Malsyuchenko.
\end{proof}

%\begin{remark} Examples~\ref{e:P+CCnotKC} and \ref{e:Corson} show that the classes of Piotrowski compact spaces and $CC$-Maslyuchenko compact spaces are incomparable (in the sense that none of them is contained in the other).
%\end{remark}

Example~\ref{e:P+CCnotKC} shows that the classes of Piotrowski and $CC$-Maslyuchenko spaces contain non-metrizable compact spaces. We do not know if the same is true for the class of $QC$-Maslyuchenko spaces.

\begin{problem} Is each $QC$-Maslyuchenko compact Hausdorff space metrizable?
\end{problem}

\begin{problem} Assume that a topological space $X$ is the union $X=A\cup B$ of two closed $QC$-Maslyuchenko subspaces. Is $X$ $CC$-Maslyuchenko? $KC$-Maslyuchenko? $QC$-Maslyuchenko?
\end{problem}

%\newpage

\end{document}